\documentclass[runningheads]{llncs}
\usepackage{graphicx}
\usepackage{xcolor}
\usepackage{relsize}
\usepackage[english]{babel}
\usepackage{ifthen}
\usepackage{amssymb}
\usepackage{hyperref}
\usepackage{ntheorem}
\usepackage{mathrsfs}
\usepackage{ cmll }
\usepackage{bigints}
\usepackage{bussproofs}

\usepackage{setspace,caption}
\usepackage{tikz-cd}
\usetikzlibrary{calc}
\usepackage{graphicx}
\usepackage{blkarray}
\usepackage{setspace}
\usepackage{color}
\usepackage{pgf}
\usepackage{tikz}
\usetikzlibrary{arrows,automata}
\usepackage[all]{xy}  
\usepackage{multicol}
\usepackage{float}

\usepackage{amsmath}
\usetikzlibrary{positioning}
\newlength{\indentationFormule}
\newlength{\indentationTotaleFormule}
\newlength{\indentationCommentaire}
\newlength{\indentationDerivation}
\newlength{\largeurLangle}
\newlength{\largeurBoiteCommentaire}

\newenvironment{deriv}[1][\leftmargini]%
{\setlength{\indentationDerivation}{#1}%
	\setlength{\indentationTotaleFormule}{\indentationFormule}
	\addtolength{\indentationTotaleFormule}{#1}
	\setlength{\largeurBoiteCommentaire}{\linewidth}
	\addtolength{\largeurBoiteCommentaire}{-\indentationFormule}
	\addtolength{\largeurBoiteCommentaire}{-\indentationCommentaire}
	\addtolength{\largeurBoiteCommentaire}{-\largeurLangle}
	\addtolength{\largeurBoiteCommentaire}{-\indentationDerivation}
	\begin{list}{}{\setlength{\leftmargin}{\indentationTotaleFormule}}
		\setlength{\baselineskip}{1.3\baselineskip}
		\item$}%
	{\hbox{}$\end{list}}  

\newcommand{\<}[1]{\\\hspace*{-\indentationFormule}\makebox(0,0)[bl]{$#1$}\hspace*{\indentationFormule}}

\newcommand{\commentaire}[1]{\hspace*{\indentationCommentaire}\langle\hspace*{.4em}%
	\begin{minipage}[t]{\largeurBoiteCommentaire}%
		#1 $\rangle$\\\hbox{}
	\end{minipage}\\[-1ex] 
}

\begin{document}
\title{ Mathematical Morphology \textit{via} Category Theory
}
%
%
%

\author{Hossein Memarzadeh Sharifipour \inst{1} \and
Bardia Yousefi\inst{2,3} }
\authorrunning{H.M. Sharifipour \& B. Yousefi}
%
\institute{Department of Computer Science, Laval University, Québec, CA \email{Hossein.MemarzadehSharifipour.1@ulaval.ca}\and
Department of Electrical and Computer Engineering, Laval University, Québec, CA
\\\and
\textit{Address:} University of Pennsylvania, Philadelphia PA 19104\\
\email{Bardia.Yousefi.1@ulaval.ca}}
\maketitle              
\begin{abstract}
Mathematical morphology contributes many profitable tools to image processing area. Some of these methods considered to be basic but the most important fundamental of data processing in many various applications. In this paper, we modify the fundamental of morphological operations such as dilation and erosion making use of limit and co-limit preserving functors within (\textit{Category Theory}). 
Adopting the well-known matrix representation of images, the category of matrix, called $\mathsf{Mat}$, can be represented as an image. With enriching $\mathsf{Mat}$ over various semirings such as \textit{Boolean} and $\mathsf(max,+)$ semirings, one can arrive at classical definition of binary and gray-scale images using the categorical tensor product in $\mathsf{Mat}$. With dilation operation in hand, the erosion can be reached using the famous \textit{tensor-hom} adjunction. This approach enables us to define new types of dilation and erosion between two images represented by matrices using different semirings other than \textit{Boolean} and $\mathsf(max,+)$ semirings. The viewpoint of morphological operations from \textit{category theory} can also shed light to the claimed concept that mathematical morphology is a model for linear logic.

\end{abstract}
\section{Introduction}
\paragraph{Mathematical morphology} is a structure-based analysis of images constructed on set theory concepts. Two main induced transformations in mathematical morphology are dilation and erosion established initially by translations, unions and intersections on subsets of euclidean spaces. These transformations were extended to complete lattices later. In a more general form, dilation and erosion are Galois connections over mappings between complete lattices.     

According to ~\cite{heijmans1990algebraic}, every operator on a complete lattice which preserves finite supremum($\vee$) is regarded as a dilation and each infimum preserving operator can be inspected as an erosion. This is the most universal definition of dilation and erosion appeared in literature up to our knowledge. Looking toward category theory viewpoint, definition of the dilation and erosion can be generalized to left and right adjoint functors that preserve co-limits and limits respectively.

The first section stares at mathematical morphology and category theory briefly. Most of the consisting material for category theory  can be found  in \cite{borceux1994handbook,opac-b1078351,borceux1994handbook2,leinster2014basic} . 
\subsection{Matrix representation of images}
An image on a computer springs up from quantization of both image space and intensities into discrete values. It is merely a 2D rectangular array of integer values. It is widely accepted to record intensity as an integer number over the interval $[0,255]$ \cite{memar2017}. Conversely, each pixel of a colored image contains three values over the interval $[0,255]$ corresponding to its RGB values. Roughly speaking, the matrix representation of an image deals closely with a function entitled as a picture function. It is a function $f$ defined on spatial variables $x,y$. Intuitively, $f(x,y)$ defines the intensity value of the pixel at point $(x,y)$. 
The following definitions correspond to binary, gray-scaled and color images.
\begin{definition}
	A binary image is a rectangular matrix with all elements values 0 or 1. 
\end{definition}
\begin{definition}
	A gray-scaled image is a rectangular matrix with values ranging within $[0,255]$.  
\end{definition}
\begin{definition}
	A color image is a 2D image which has vector of 3 values at each spatial point or pixel. 
\end{definition} 

\subsection{Mathematical morphology} 
Dilation and erosion constitute the basic operations which construct the backbone for clarifying other widely used operations such as opening, closing, hit-miss and a few others. They have their roots in set theory and boil from a set theoretical view point. These transformations are defined on elements of two sets called the source and structuring element respectively. The structuring element is generally much smaller sized comparing to the image that it acts on. It functions as a pattern probing the source image, targeting at finding its structure. First we define dilation for  binary images as follows \cite{heijmans1990algebraic}.
\begin{definition}
	If $A,B$ are two sets determining the source image and the structuring element respectively, dilation of $A$ and $B$ is defined by \cite{haralick1987image} and  \cite{serra1983image} as   
	\begin{equation}  
	A \oplus B=\{a+b| a \in A , b \in B\}, 
	\end{equation}
	
\end{definition}
This operation is also called the Minkowski sum. The pertinence of dilation in image analysis area varies from image expanding to filling holes. Erosion is the dual of dilation defined by: 
\begin{definition}
	The erosion of a set $A$ with a structuring element $B$ is: \cite{serra1983image,heijmans1990algebraic}
	\begin{equation}
	A \ominus B= \{ x \in \mathbb{Z}^2| \text{ for every  } b \in B , \text{ there exists  an } a \in A \text{ such that }  x=a-b\}
	\end{equation}
	Erosion of two sets $A,B$ can also be defined as 
	\begin{equation}  A \ominus B=\{h \in \mathbb{Z}^2| B_h  \subseteq A \}   \end{equation}
	where $B_h=\{ b+h| b \in B\}$ is the translating of $B$ along the vector $h$ and reflection of the set $B$ with respect to origin is defined like 
	\begin{equation} \breve{B}=\{x \in \mathbb{Z}^2|  \text{ for some } b \in B, x =-b\} \end{equation}
\end{definition}
duality of dilation and erosion means that erosion can be written in terms of the dilation:
\begin{equation}
(A \ominus B) ^c=A^c \oplus \breve{B}
\end{equation}
where $\breve{B}$ has been defined before (3.13). In other words, dilating the foreground is the same as eroding the background, but the structuring element reflects
between the two. Likewise, eroding the foreground is the same as dilating the background.

Dilation and erosion are defined for gray scale images in a different way. \\
Dilation and erosion  of $f$ where $f: F \rightarrow \mathbb{Z}, F \subseteq \mathbb{Z}^2$ is a function that maps $(x,y) \in \mathbb{Z}^2$ to gray scale value of pixel at $(x,y)$ with structuring element $B$ is denoted as \begin{equation}
(f \oplus B)(x,y)=\text{max}_{(s,t) \in B} \{f(x-s,y-t)\} 
\end{equation}
\begin{equation}
(f \Theta B)(x,y)=\text{min}_{(s,t) \in B} \{f(x+s,y+t)\} 
\end{equation}
The following more general definition of dilation and erosion can be found in \cite{heijmans1990algebraic}. 
\begin{definition}
	Let $\mathscr{L}$ be a complete lattice and $\mathscr{E}_1,\mathscr{E}_2$ be arbitrary sets. The operator $\delta: \mathscr{L}^{\mathscr{E}_1} \longrightarrow \mathscr{L}^{\mathscr{E}_2}$ is a dilation if and only if for every $x \in \mathscr{E}_1$ and $y \in \mathscr{E}_2$ there exists a $\delta_{y,x}: \mathscr{L} 
	\longrightarrow \mathscr{L}$ such that for $F_1 \in \mathscr{L}^{\mathscr{E}_1}$ and $ y \in \mathscr{E}_2$, \[ \delta(F_1)(y) = \bigvee_{x \in \mathscr{E}_1} \delta_{y,x}(F_1(x)) \]
	
	The erosion  $\epsilon: \mathscr{L}^{\mathscr{E}_1} \longrightarrow \mathscr{L}^{\mathscr{E}_2}$  is given by:  \[ \epsilon(F_2)(x) = \bigwedge_{y \in \mathscr{E}_2} \epsilon_{y,x}(F_2(y)) \]
\end{definition}
\begin{definition}
	Let $(A,\leq),(B,\leq)$ be two partially ordered sets with two mappings, $F:A \longrightarrow B$ and $U:B \longrightarrow A$. A monotone Galois connection between $F,U$ is for all $x \in A$ and $y \in B$ : \[ F(x) \leq y \Leftrightarrow x \leq U(y) \]
\end{definition}
$F$ is referred as the left adjoint and $U$ as an right adjoint. Monotone Galois connections are entitled as adjunctions in literature \cite{erne2004adjunctions,shmuely1974structure} likewise. The other variety of Galois connections called antitone Galois connections emerges in literature as follows.
\begin{definition}
	Let $(A,\leq),(B,\leq)$ be two partially ordered sets. Let $F:A \longrightarrow B$ and $U:B \longrightarrow A$. $F,U$ are an antitone Galois connection if for all $x \in A$ and $y \in B$ : \[ y \leq F(x) \Leftrightarrow x \leq U(y) \]
\end{definition}
\begin{theorem}
	\[ \delta(F_1) \leq F_2 \Leftrightarrow \epsilon(F_2) \leq F_1 \]
\end{theorem} 
This theorem confirms that dilation and erosion engage in a monotonic Galois connection.
\subsection{Category theory}
Category theory is an effort for generalizing and simplifying many properties of mathematical systems by denoting them with objects and arrows. Each arrow $f:A \to B$ represents a function from an object $A$ to another object $B$. A category is small if its set of objects and arrows are small.

A contravariant functor $F:\mathcal{A}^{OP} \to \mathcal{B}$ maps every object  $A \in \mathcal{A}$ to $F(A) \in \mathcal{B}$ and there exist a mapping $\mathcal{A}(A,A^\prime) \to \mathcal{B}(FA,FA^\prime)$ for each mapping $f:A \to A^\prime \in \mathcal{A}$. Reminding that $\mathcal{A}(A,A^\prime)$ is an arrow between the objects, these data are subject to following two conditions: 
\begin{itemize}
	\item 
	any two morphism $f \in \mathcal{A}(A,A^\prime)$ and $g \in \mathcal{A}(A^\prime,  A''$) can be decomposed by $F(g \circ f)=F(f) \circ F(g)$. 
	\item  For any object $A \in \mathscr{A}$ , $F(1_A)=1_{FA}$. 
\end{itemize}
A contravariant  functor reverses the direction of arrows. For example $f:A \to B$  gets $F(f):f(B) \to f(A)$.

Conversely, a covariant functor $F: \mathcal{C} \to \mathcal{D}$ preserves the direction of arrows. Everything is the same as the contravariant functor except : $F(f \circ g)=F(f) \circ F(g)$ for arrows $f,g$ in $\mathcal{C}$.

\subsubsection{Natural transformations}
\begin{definition}
	Let $\mathcal{A}$ and $\mathcal{B}$ be categories and $F,G$ two functors $F,G:\mathcal{A} \to \mathcal{B}$. A natural transformation between $F,G$ is an arrow $\alpha_x:F(x) \to G(x)$ for any object $X \in \mathcal{A}$ such that for any arrow $f: X \to Y \in  \mathcal{A}$, the diagram depicted in figure \ref{fig1} commutes,
	\begin{figure}[!h]
		\begin{center}
			\label{def:wedge}	
			\begin{tikzpicture}
			\matrix (m) [matrix of math nodes,row sep=4em,column sep=4em,minimum width=4em]
			{
				F(X)	& F(Y)  \\  G(X) &    G(Y) \\};
			\path[-stealth]
			(m-1-1)   	   edge [sloped, above] node  {$F(f)$} (m-1-2)
			(m-1-1)   	   edge [sloped, above] node  {$\alpha_X$} (m-2-1)
			(m-2-1)   	   edge [ sloped, above] node  {$ G(f) $} (m-2-2)
			(m-1-2)   	   edge [ sloped, above] node  {$\alpha_Y $} (m-2-2);
			\end{tikzpicture}		
		\end{center}
		
		\caption {Natural transformations }
	     \label{fig1}
	\end{figure}
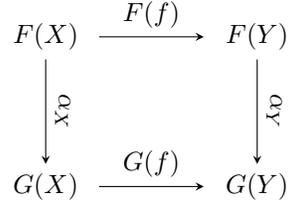
\end{definition}
Natural transformations resemble functor isomorphisms. 
\subsubsection{(Co)-limits}
\begin{definition}
	Given a functor $F: \mathcal{A} \to \mathcal{B}$, a cone of $F$ is an object $O \in \mathcal{B}$ together with a family of arrows $\phi_x:O \to F(x)$ for each $x \in \mathcal{A}$ such that for each arrow $f: x \to y $ in $\mathcal{A}$ we have $Ff\circ \phi_x=\phi_y$ according to figure \ref{fig5}.
	
	\begin{figure}	
		\begin{center}
			\begin{tikzpicture}[descr/.style={fill=white}]
			\matrix(m)[matrix of math nodes, row sep=3em, column sep=2.8em,
			text height=1.5ex, text depth=0.25ex]
			{&O\\F(X)&F(Y)\\};
			\path[->,font=\scriptsize]
			(m-1-2) edge node[above left] {$\phi_x$} (m-2-1)
			edge node[descr] {$\phi_y$} (m-2-2);					
			\path[->]
			(m-2-1) edge node [below]{$F(f)$} (m-2-2);
			\end{tikzpicture}
		\end{center}
		\caption {Diagram of a cone}
		\label{fig5}
	\end{figure}
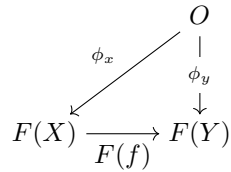
    
\end{definition}
\begin{definition}
	A limit of a functor $F: \mathcal{A} \to \mathcal{B}$, is a universal cone $(L,\phi_x)$  such that for every other cone $(N,\psi_x)$ of $F$ there is a unique arrow $u: N \to L$ such that $\psi_x= \phi_x \circ u$ for every $X$ in $\mathcal{B}$. \\
        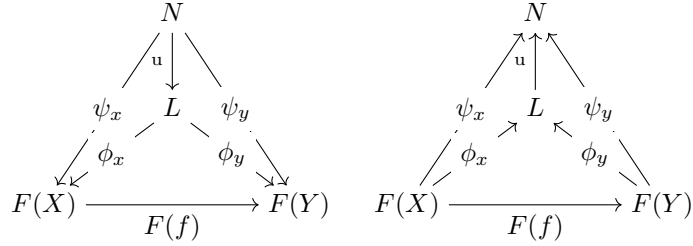
\begin{figure}[h]
        \centering	
		\begin{tikzpicture}[descr/.style={fill=white}]
		\matrix(m)[matrix of math nodes, row sep=2.2em, column sep=2.5em,
		text height=1.5ex, text depth=0.25ex]
		{&N\\&L\\ F(X)&& F(Y)\\};
		\path[->,font=\scriptsize]
		(m-1-2) edge node[ left] {u} (m-2-2);
		\path[->]
		(m-3-1) edge node [below]{$F(f)$} (m-3-3);
		\path[->]
		(m-1-2) 
		edge node[descr] {$\psi_x$} (m-3-1)
		edge node[descr] {$\psi_y$} (m-3-3);
		\path[->]
		(m-2-2) 
		edge node[descr] {$\phi_x$} (m-3-1)
		edge node[descr] {$\phi_y$} (m-3-3);
		\end{tikzpicture}
		\begin{tikzpicture}[descr/.style={fill=white}]
		\matrix(m)[matrix of math nodes, row sep=2.2em, column sep=2.5em,
		text height=1.5ex, text depth=0.25ex]
		{&N\\&L\\ F(X)&& F(Y)\\};
		\path[->,font=\scriptsize]
		(m-2-2) edge node[left] {u} (m-1-2);
		\path[->]
		(m-3-1) edge node [below]{$F(f)$} (m-3-3);
		\path[->]
		(m-3-1) 
		edge node[descr] {$\psi_x$} (m-1-2);
		\path[->]
		(m-3-3) 
		edge node[descr] {$\psi_y$} (m-1-2);
		\path[->]
		(m-3-1) 
		edge node[descr] {$\phi_x$} (m-2-2);
		\path[->]
		(m-3-3)  edge node[descr] {$\phi_y$} (m-2-2);
		\end{tikzpicture}
    	\caption{ a) Diagram of a limit b) Diagram of a co-limit}
    \label{fig4}
\end{figure}	
\end{definition}
\begin{definition}
	Dual to limit, a co-limit of a functor $F: \mathcal{A} \to \mathcal{B}$, is a universal co-cone $(L,\phi_x)$  such that for any other co-cone $(N,\psi_x)$ of $F$, there exists a unique arrow $u: L \to N$ such that $\psi_x=u \circ \phi_x$ for every $X$ in $\mathscr{B}$. Figure \ref{fig4} (b) illustrates the diagram of a co-limit. A functor $F: \mathcal{A} \to \mathcal{B}$ is called small if $\mathcal{B}$ is a small category. (co)-limits over small functors are called small.	
\end{definition} 
Symbols $\mathop{\lim_{\longrightarrow}}$ and $\mathop{\lim_{\longleftarrow}}$ are used to denote lim and co-limit often in literature.
A category is called (co)-complete if it contains all small (co)-limits. Symbols $\mathop{\lim_{\longrightarrow}}$ and $\mathop{\lim_{\longleftarrow}}$ are used to denote lim and co-limit often in literature.
\subsubsection{(Co)-ends}
(co)-ends are useful notions inspired from calculus. Particularly, an end resembles an infinite product whereas a co-end imitates the idea of an infinite sum or integral. 

(co)-ends are special (co)-limits defined on functors of the form $F: \mathcal{C}^{OP} \times \mathcal{C} \to \mathcal{D}$. Defining (co)-wedge is essential since a (co)-end is a universal (co)-wedge. 
	
A wedge of a functor $F: \mathcal{C}^{OP} \times \mathcal{C} \to \mathcal{D}$ is an object $O \in \mathcal{D}$ with an arrow $\omega_c: O \to F(C,C)$ for any object $C \in \mathcal{C}$. The universal property of a wedge enforces that for any arrow $C^\prime \to C$ for $C,C^\prime \in \mathcal{C}$, the diagram illustrated in \label{def:wedge} (a) commutes. 

	Conversely, co-ends are defined by natural co-wedges. A co-wedge for a functor $F:\mathcal{C}^{OP} \times \mathcal{C} \to \mathcal{D}$ is an object $O$ in $\mathcal{D}$ along with an arrow $\omega_c:  F(C,C) \to O$ such that for any arrow $t: C^\prime \to C$ in $\mathcal{C}$, the diagram illustrated in \label{def:wedge} (a) commutes.  
\begin{figure}[!h]
	\begin{center}
	\label{def:wedge}	
		\begin{tikzpicture}
		\matrix (m) [matrix of math nodes,row sep=4em,column sep=4em,minimum width=4em]
		{
			O	& F(C,C)   \\  F(C^\prime,C^\prime) &   F(C,C^\prime) \\};
		\path[-stealth]
		(m-1-1)   	   edge [sloped, above] node  {$\omega_c$} (m-1-2)
		(m-1-1)   	   edge [sloped, above] node  {$\omega_c^\prime$} (m-2-1)
		(m-2-1)   	   edge [ sloped, above] node  {$ F(t,1) $} (m-2-2)
		(m-1-2)   	   edge [ sloped, above] node  {$F(1,t) $} (m-2-2);
		\end{tikzpicture}		
		\begin{tikzpicture}
		\matrix (m) [matrix of math nodes,row sep=4em,column sep=4em,minimum width=4em]
		{
			F(C,C^\prime)	& F(C,C)   \\  F(C^\prime,C^\prime) &   O \\};
		\path[-stealth]
		(m-1-1)   	   edge [sloped, above] node  {$F(1,t)$} (m-1-2)
		(m-1-1)   	   edge [sloped, above] node  {$F(t,1)$} (m-2-1)
		(m-2-1)   	   edge [ sloped, above] node  {$ \omega_c ^\prime $} (m-2-2)
		(m-1-2)   	   edge [ sloped, above] node  {$\omega_c $} (m-2-2);
		\end{tikzpicture}
		\end{center}

	\caption{a)Diagram of a wedge b)Diagram of a co-wedge}
\end{figure}
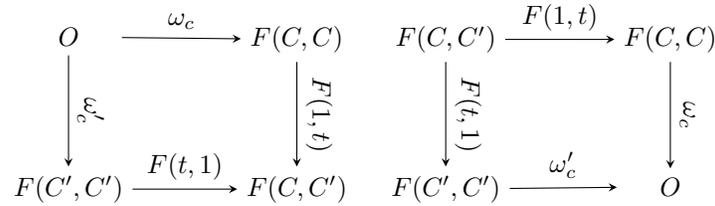

The abusing integral notation for denoting (co)-ends stems from the work of N.Yoneda  \cite{loregian2015co} while he came up with functors $\mathcal{C}^{OP} \times \mathcal{C} \to \mathsf{Ab}$. The subscripted integral notation $\int_C F(C,C)$ denotes an end of a functor $F:\mathcal{C}^{OP} \times \mathcal{C} \to \mathcal{D}$ whereas the superscripted  integral notation $\int^C F(C,C)$ demonstrates a Co-end for $F$.

A helpful property of ends which makes them so useful is their capability of representing natural transformations. This  can be expressed by the following theorem,
\begin{theorem}
	Given two functors $F,G: \mathcal{C} \to \mathcal{D}$, the set of natural transformations between $F,G$ denoted by $[C,D](F,G)$ equals to $[C,D](F,G)=\bigintssss \limits_{c \in \mathscr{C}}  D(F(c),G(c))$.
\end{theorem}
\begin{proof}
	Suppose $\bigintssss \limits_{c \in \mathscr{C}}  D(F(c),G(c))$ includes morphisms $h(c):F(c) \to G(c)$ in $\mathscr{D}$. For any other morphism $f: c \to d$ in $\mathscr{C}$, The following diagram commutes as a consequence of an end properties.	
	\begin{figure}[!h]
		\begin{center}
			\begin{tikzpicture}
			\matrix (m) [matrix of math nodes,row sep=3em,column sep=4em,minimum width=5em]
			{
				F(c)	& F(d) &  \\ G(c)  &  G(d)  \\};
			\path[-stealth]
			(m-1-1)   	   edge [sloped, above] node  {$h(c)$} (m-2-1)
			(m-1-2)   	   edge [ sloped, above] node  {$ h(d) $} (m-2-2)
			(m-1-1)   	   edge [ sloped, above] node  {$ F(f) $} (m-1-2)
			(m-2-1)   	   edge [ sloped, above] node  {$ G(f) $} (m-2-2);
			\end{tikzpicture}
		\end{center}
	\end{figure} 
	This is by definition a natural transformation $F \Rightarrow G$.
\end{proof}	
Some practical properties of ends that will be used later are:
\begin{equation}
\int \limits_{ a \in A}  \int \limits_{b \in B} F(a,a,b,b) = \int \limits_{b \in B}  \int \limits_{a \in A} F(a,a,b,b) = \int \limits_{(a,b) \in A\times B} F(a,a,b,b)
\end{equation}
\begin{equation}
\int \limits_{A} [C,F(A,A)] \simeq [C, \int \limits_{A} F(A,A)] \hspace*{52mm}
\label{eq1}
\end{equation}
\begin{equation}
\int \limits_{A} [F(A,A), C] \simeq [\int \limits^{A} F(A,A),C] \hspace*{52mm}
\label{eq2}
\end{equation}
\begin{definition}
A functor $\mathcal{C}, \mathcal{D}$ along with functors $F: \mathcal{D} \to \mathcal{C}$ and $G: \mathcal{C} \to \mathcal{D}$, 
$F$ is assumed to preserve co-limits, if $\mathop{\lim_{\longrightarrow}}_{i}X_i$ which exists in $\mathcal{C}$ for a functor $X: \mathcal{I} \to \mathcal{C}$, $F(\mathop{\lim_{\longrightarrow}}_{i}X_i) \simeq \mathop{\lim_{\longrightarrow}}_{i}F(X_i)$. Conversely, $G$ is assumed to preserve all small limits if it forces that the limit $\mathop{\lim_{\longleftarrow}}_{i}X_i$ for a functor $X: \mathscr{I} \to \mathscr{C}$ in $\mathscr{C}$ if exists, $G(\mathop{\lim_{\longleftarrow}}_{i}X_i) \simeq \mathop{\lim_{\longleftarrow}}_{i}G(X_i)$. 
\label{preservingLimits}
\end{definition}	
	
\subsection{Adjunctions}
One of the main intentions of mathematics is to compare two models. One way to express that two models or objects are similar is by equality. However, equality is too much strong in many cases. Another similarity observation in some cases is isomorphism.
Isomorphism is a weaker notion comparing with equality. Two categories $\mathcal{C},\mathcal{D}$ are isomorphic if there exist two functors $R: \mathcal{C} \to \mathcal{D}$ and $L:\mathcal{D} \to \mathcal{L}$ such that $L \circ R=\mathsf{id}_{\mathcal{C}}$ and $R \circ L=\mathsf{id}_{\mathcal{D}}$. However, the notion of isomorphism is also too ambitious to expect in many cases. 

Adjunction weakens even the requirements needed by isomorphism for two categories by just asking a one way natural transformation $\eta:\mathsf{id}_{\mathcal{D}} \Rightarrow R \circ L$ and another natural transformation expressing $\epsilon:L \circ R \Rightarrow \mathsf{id}_{\mathcal{C}}$. In the language of category theory $\eta$ is called the unit and $\epsilon$ is called the co-unit of the adjunction.  Functor $L$ is noted as the left adjoint to functor $R$ and $R$ is called the right adjoint to $L$. 

One can also express the adjunctions in terms of triangular identities \cite{borceux1994handbook,borceux1994handbook3} depicted by diagrams in figure \ref{adjoint}: 

\begin{figure}[h]
	\centering
	\begin{tikzpicture}[node distance=3.1cm, auto]
	
	\pgfmathsetmacro{\shift}{0.3ex}
	
	\node (R) {$L$};
	\node (P) [right of=R]{$L \circ R \circ L$};
	\node (B) [below of=P] {$L$};
	
	\draw[transform canvas={yshift=0.5ex},->] (R) --(B) node[above,midway] { };
	\draw[transform canvas={yshift=-0.5ex},->](B) -- (R) node[below,midway] { }; 
	
	\draw[transform canvas={yshift=-0.5ex},->](R) -- (P) node[above,midway] { $L \circ \eta$}; 
	\draw[transform canvas={xshift=-0.5ex},->](P) -- (B) node[right,midway] { $\epsilon \circ L$};

	\end{tikzpicture} 
	\\
	\begin{tikzpicture}[node distance=3.1cm, auto]
	
	\pgfmathsetmacro{\shift}{0.3ex}
	
	\node (R) {$R$};
	\node (P) [right of=R]{$R \circ L \circ R$};
	\node (B) [below of=P] {$R$};
	
	\draw[transform canvas={yshift=0.5ex},->] (R) --(B) node[above,midway] { };
	\draw[transform canvas={yshift=-0.5ex},->](B) -- (R) node[below,midway] { }; 
	
	\draw[transform canvas={yshift=-0.5ex},->](R) -- (P) node[above,midway] { $ \eta \circ R$}; 
	\draw[transform canvas={xshift=-0.5ex},->](P) -- (B) node[right,midway] { $R \circ \epsilon$};

	\end{tikzpicture} 
	\caption{Triangle diagrams(Adjunctions)}
	\label{adjoint}
\end{figure}
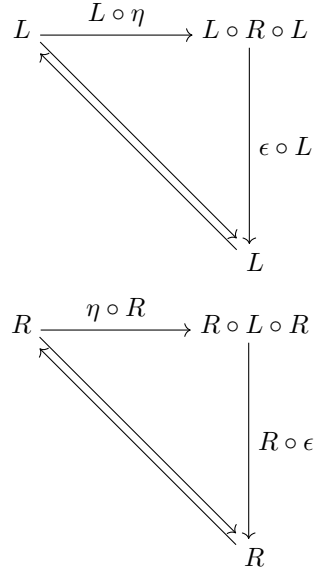
\textbf{Example.}\\
Adjunctions in the category of preorders corresponds to functors $F: L_1 \to L_2 ^{OP} $, $G: L_2 ^{OP} \to L_1$ between two preorders $L_1, L_2$.  $L$ is the left adjoint of $G$ iff for any $p \in L_1, q \in L_2$,\begin{center} $q \leq F(p) \implies p \leq G(q) $ \end{center}
The adjunction in category of complete lattices called Galois connections plays a crucial role in mathematical morphology area. To express more, any left adjoint in the category of complete lattices is a dilation and its right adjoint is an erosion consequently.

A major property of adjoints which is widely used in category theory is that they preserve (co)-limits. A left adjoint preserves co-limits whereas a right adjoint preserves limits.

\subsection{The Yoneda lemma}
The Yoneda lemma is a major and applicative result of category theory \cite{riehl2017category}. It allows us to embed any category into the category of contravariant functors stemming from that category to the category of sets. The Yoneda lemma makes the life easier by suggesting that one can investigate functors from a small category to the category of sets instead of investigating directly on it. In many cases the former inspection is much more easier. 
\begin{definition}
	Consider a functor $F:\mathcal{A} \to \mathsf{Set}$ from an arbitrary category $\mathcal{A}$ to the category of sets,  an object $A \in \mathcal{A}$ and the corresponding functor $\mathcal{A}(A,-): A \to \mathsf{Set}$. There exists a bijective correlation, \[ \mathsf{nat}(\mathcal{A}(A,-),F) \simeq FA\]
\end{definition}
 The main idea inherited in Yoneda lemma is that every information we need about an object $A \in \mathcal{C}$ is encoded in $\mathcal{C}[-, A]$. Yoneda lemma can be expressed by co-ends also. 
 Let $F:\mathcal{C}^{OP} \to \mathsf{Set}$ and $G:\mathcal{C} \to \mathsf{Set} $ be functors. The following formulas express the Yoneda lemma.
 \begin{equation*}
 F \simeq 	\int \limits^{A \in \mathcal{C}} FA \times \mathcal{C}[-,A] 
 \end{equation*}
 \begin{equation*}
 G \simeq 	\int \limits^{A \in \mathcal{C}} GA \times \mathcal{C}[A,-] 
 \end{equation*}
 
\subsubsection{Monoidal categories}
	\begin{definition}
	A category $\mathcal{C}$ is monoidal if it is  equipped with a tensor product $\otimes$ that satisfies some conditions. Roughly speaking $ \otimes: \mathcal{C} \times \mathcal{C} \to \mathcal{C}$ is a functor satisfying:
	\begin{itemize}
		\item $(A \otimes B) \otimes C \rightarrow A \otimes (B \otimes C)$ (Associativity isomorphism law)
		\item There exist an identity object $I$ satisfying: $\lambda_A:A \otimes I  \rightarrow A$ and $\rho_A:I \otimes A \rightarrow A$
		\item The following two commutations hold:	
	\end{itemize}
	\begin{center}
		\begin{tikzpicture}
		\matrix (m) [matrix of math nodes,row sep=4em,column sep=3em,minimum width=2em]
		{
			& (A \otimes (B \otimes C)) \otimes D \\    ((A \otimes B) \otimes C) \otimes D &  & A \otimes ( (B \otimes C) \otimes D)  \\
			(A \otimes B) \otimes (C \otimes D)  & & (A \otimes B ( \otimes (C \otimes D))  \\ };
		\path[-stealth]
		(m-2-1)   	   edge [sloped, above] node  {$\alpha_{A,B,C} \otimes 1_D$} (m-1-2)
		(m-1-2)   	   edge [sloped, above] node  {$\alpha_{A,B \otimes C,D } $} (m-2-3)
		(m-2-1)   	   edge [ right] node  {$\alpha_{A \otimes B, C,D } $} (m-3-1)
		(m-2-3)   	   edge [ right] node  {$ 1_A \otimes \alpha_{B, C,D } $} (m-3-3)
		(m-3-1)   	   edge [above] node  {$\alpha_{A ,B, C \otimes D } $} (m-3-3);
		\end{tikzpicture}
		\begin{tikzpicture}
		\matrix (m) [matrix of math nodes,row sep=4em,column sep=3em,minimum width=2em]
		{
			(A \otimes I) \otimes B & &     A \otimes (I \otimes B) \\  & A \otimes B \\};
		\path[-stealth]
		(m-1-1)   	   edge [sloped, above] node  {$\alpha_{A,I,B}$} (m-1-3)
		(m-1-1)   	   edge [sloped, above] node  {$\lambda_A \otimes 1_B $} (m-2-2)
		(m-1-3)   	   edge [ sloped, above] node  {$1_A \otimes \rho_B $} (m-2-2);
		\end{tikzpicture}
	\end{center}
\end{definition}
A monoidal category is left closed if for each object $A$ the functor $B \longmapsto A \otimes B$ has a right adjoint $ B \longmapsto ( A \Rightarrow B) $. This means that the bijection $\mathsf{Hom}_\mathcal{C}(A\otimes B, C) \cong \mathsf{Hom}_\mathcal{C}(A, B \Rightarrow C)$ between the Hom-sets called currying exists. 

Dually, the monoidal category $\mathcal{C}$ is right closed if the functor $ B \longmapsto B \otimes A$  admits a right adjoint. In a symmetric monoidal category the notions of left closed and right closed coincide.
\subsection{Enriched categories}
Generally, the arrows between two objects A,B in category $\mathcal{C}$ illustrated by $\mathcal{C}[A,B]$ are a set. The notion of enrichment extends that structure from merely a set to more fruitful structures \cite{kelly1982basic,day1969enriched}. For instance the set of arrows between two objects may have an abelian structure meaning that it is possible to add the arrows between two objects. A restriction imposed on a category $\mathcal{V}$ on which the arrows are enriched over it is that it should have a monoidal structure. More formally,
\begin{definition}
	Let $\mathcal{V}$ be a symmetric monoidal category. A category enriched over $\mathcal{V}$ called a $\mathcal{V}$-category consists of:
	\begin{itemize}
	\item For all objects $A,B \in \mathcal{C}$, the set of arrows from $A$ to $B$ is an object $O \in \mathcal{V}$.
	\item For all object $A,B,C \in \mathcal{C}$, there exist composition of arrows in $\mathcal{V}$ such that $c_{A,B,C}: \mathcal{C}[A,B] \times \mathcal{C}[A,B] \to \mathcal{C}[A,B]$.
	\item For each object $A \in \mathcal{C}$, an identity arrow exists in $\mathcal{V}$ such that $I \to \mathcal{C}(A,A)$.	
\end{itemize}	
	
\end{definition}	
\section{Categorical dilation and erosion}

Considering the well-known notion of dilation in the category of complete lattices ($\mathsf{CompLat}$), we may generalize it to other categories as follows:
\begin{definition}
  A dilation is a co-limit preserving functor whereas an erosion is a limit preserving functor. 
	\label{dilation}
\end{definition}

We expressed previously that left/right adjoint functors preserve co-limits/limits respectively. Thus, the claim that left/right adjoints are a major source of dilation and erosion functors is pretty precise. 

Another concept that we will utilize it for defining morphological operations on matrix representation of images is the Day convolution \cite{day1970closed,im1986universal}. Let $\mathcal{V}$ be a complete and co-complete small symmetric monoidal category. Let the enriched Yonedda embedding functor to be $ \mathcal{C} \to [\mathcal{C}^{OP}, \mathcal{V}]$. The intuition of Day convolution is that a monoidal structure on $ \mathscr{C}$ brings the monoidal structure on $[\mathscr{C}^{OP}, \mathscr{V}]$. 

The enriched co-Yoneda lemma states that any $\mathscr{V}$-enriched functor $[\mathscr{C}^{OP},\mathscr{V}]$  is canonically isomorphic to a co-end of representables. This can be expressed as $F(c)  \simeq \mathlarger{\int^{c}F(c) \otimes \mathcal{C}(-,c)}$.
 Taking two functors $F,G: \mathcal{C}^{OP} \to \mathcal{V}$, define their multiplication as:
\[ F \ast G=\mathlarger{\int^{c}F(c)} \otimes \mathcal{C}(-,c) \ast \mathlarger{ \int^{b}G(b) \otimes \mathcal{C}(-,b)} \] 
Assuming the multiplication operation interchanges properly with the co-end, we get: \[ F \ast G=\mathlarger{\int^{c,b}F(c) \otimes G(b) \otimes (\mathcal{C}(-,c) \ast \mathcal{C}(-,b)}) \] 
Forcing the Yoneda embedding $ \mathcal{C} \to [\mathcal{C}^{OP}, \mathcal{V}]$ be strongly monoidal yields to:
\begin{equation}
F \ast G  \\ \simeq \mathlarger{\int^{c,b}}F(c) \otimes G(b) \otimes  \mathcal{C}(-,c \otimes b).
\end{equation}  
\begin{definition}
	A \textit{semiring} category $\mathcal{S }$ is a category with two operations of addition and multiplication shown as $(\oplus,\cdot)$ such that $(\mathcal{S},\oplus,0)$ is monoidal and symmetric, $(\mathcal{S},\dot,1)$ is monoidal. Left and right distributivity of multiplication over addition is expressed by natural isomorphisms as:
\begin{flalign*}
A \cdot (B \oplus C) \to (A \cdot B) \oplus (A \cdot C) \\
(B \oplus C) \cdot A \to (B \cdot A) \oplus (C \cdot A)
\end{flalign*}
Evidently, a semiring category is a ring in which the elements lack their inverses for addition. The max-plus category may be defined by $a \oplus b = \max \{a,b\}$ and $a \cdot b = a +b$ along with $ +\infty$ and $0$ acting as unit objects for addition and multiplication.  
\end{definition}
Notions such as group, ring, semirings and many other varieties of algebraic structures are categorized via Lawvere theory \cite{lawvere1963functorial}.  If $\mathsf{T}$ is a Lawvere theory and $\mathcal{C}$ is a category with finite products such as $\mathsf{set}$, a functor $F: \mathsf{T}^{OP} \to \mathcal{C}$ creates a full subcategory called a model of $\mathsf{T}$. The model of Lawvere theories when $\mathcal{C}$ is category of sets is referred as $\mathsf{T}$-algebras. $\mathsf{T}$-algebras built on complete and co-complete category $\mathcal{C}$ which is in most cases category of sets are complete and cocomplete. Literally, the forgetful functors from a $\mathsf{T}$-algebra to set category creates and preserves limits. 
\begin{definition}
Given two matrices $R_{(m,n)}$ and $S_{(p,q)}$, their tensor product known also as Kronecker product is defined by:
\[ R \otimes S = \begin{bmatrix}
r_{1,1}S & \cdots & r_{1,n}S          \\[0.3em]
\vdots &            & \vdots \\[0.3em]
r_{m,1}S & \cdots & r_{m,n}S
\end{bmatrix}\]	
\end{definition}	
	
\begin{definition}
The category of $\mathsf{Mat}$ contains the set of integers as objects with arrows between two object $m,n$ as $ m \times n$ matrices with the matrix multiplication as composition of arrows. $\mathsf{S-mat}$ is deducted from enriching $\mathsf{Mat}$ on a semiring like $S$. In other words, $\mathsf{S-mat}$ has sets as objects and $\mathsf{s}$-valued matrices $m \times n \to \mathsf{S}$ as morphisms. For instance, if  $S=(\{0,1\}, \vee, \wedge ) $ is the Boolean semiring, then the $\mathsf{S-mat}$ is exactly the well-known category of $\mathsf{Rel}$. 
\end{definition}

The $\mathsf{V-mat}$ category with the array multiplication, Kronecker product behaving as composition and tensor product respectively with the identity matrix constitute a monoidal category. 
Let $\mathcal{X}$ be the discrete category containing tuples of integers as objects with arrows $\mathcal{X}((m,n),(m,n))	= \mathbb{Z}$. Let us define $F,G$ as two dimensional matrices over a semiring $S$. Eventually, $\mathcal{X}$ will act for indexing the two matrices.
Day convolution of $F,G$ can be defined by:
\begin{equation}
F \ast G  \\ \simeq \mathlarger{\int^{(m,n),(p,q)}}F_{(m,n)} \otimes G_{(p,q)} \otimes  \mathcal{C}(-,(m,n) \otimes (p,q)).
\label{Dayconvolution}
\end{equation}
Defining the tensor product $(m,n) \otimes (p,q)$ on the discrete category $\mathscr{X}$ as $(m,n) \otimes (p,q)=(m+p,n+q)$ brings it a monoidal structure. Thus, \ref{Dayconvolution} can be written as:
\begin{equation}
(F \ast G)_{(r,s)} = \bigoplus F_{(m,n)} \cdot G_{(p,q)}.
\label{fig10}
\end{equation} 
where $ m+p =r \wedge n+q =s$. 

assuming $F,G$ as the source image and the structuring element, $(F \ast G)$ can be defined as their dilation. The following example illustrates dilation of two binary images. 

Example: \\[0.2cm] Given $F=\begin{pmatrix}  

 0& 1 & 0 \\  0 & 0 &1 \\  1 &1 &0
\end{pmatrix}$ and $ G=\begin{pmatrix}
 0& 1 & 0 \\  1 & 0 &0 \\  0 &0 &0
\end{pmatrix} $, 
Assuming $F,G$ defined on Boolean semiring, one can derive the formula from \ref{fig10} using $\vee$ and $\wedge$ as addition and multiplication respectively. So
$F \ast G$ can be calculated like
\begin{equation}
(F \ast G)[r,s] = \bigvee (F[m,n] \wedge G[p,q]) 
\label{fig12} 
\end{equation}
where $ m+p =r \wedge n+q =s$. 
 \\[0.3cm] 
Hence, the following matrix will be induced by \ref{fig12}\\
$
F \ast G= \begin{pmatrix}
0 &0 &1&0 &0 \\
0&1&0&1&0 \\
0&1&1&0&0 \\
1&1&0&0&0\\
0&0&0&0&0 \\
\end{pmatrix} 
$. 

It should be noted that this result corresponds exactly to  dilating binary image $F$ with the structuring element $G$ using well-known existing formulation. However, by migrating from Boolean to max-plus semiring, formulation of gray-scaled images dilation is calculated by the following Day convolution:
\begin{equation}
(F \ast G)_{(r,s)} = \max (F_{(m,n)} + G_{(p,q)}) .
\end{equation} 	
Other known morphological operation than can be derived from \ref{fig10} is the fuzzy dilation first appeared in \cite{de1998fuzzy}. For that intention we need to use the semiring with $ a \oplus b = \max(a,b)$ and $a \cdot b = \min (a,b)$ denoted as the $\mathsf{min-max}$ to get the formula:
\begin{equation}
(F \ast G)_{(r,s)} = \max ( \min(F_{(m,n)},  G_{(p,q)})) .
\end{equation}
in which $r=m+p \wedge s=n+q$.
We can concentrate on celebrated tensor-hom adjunction for extracting the expression of erosion from dilation. Given a closed monoidal category $\mathcal{C}$, tensor-hom adjunction states that for an object $A \in \mathcal{C}$, tensor product $- \otimes A$ is the left adjoint with the internal hom functor $[A,-]$. This can be expressed by:
\begin{equation}
\mathcal{C}[A \otimes B,C] \simeq \mathcal{C}[A,[ B,C]]
\label{hom-tensor}
\end{equation}
The morphism $\mathcal{C}[A \otimes B,C] \to \mathcal{C}[A,[ B,C]]$ is nominated as currying in literature. Intuitively, currying is achieved by $\tau: m \otimes n \to \tau (m)(n)$. The hom-tensor adjunction \ref{hom-tensor} is used to calculate the right adjoint which is equal to erosion.

\begin{deriv}
   \mathcal{V} [F \ast G, E]
	\<=
	\commentaire{Natural transformations representation by ends} 
	\bigintsss \limits_C [(F \ast G)C, EC]
	\<=
	\commentaire{Definition of Day convolution}
	\bigintsss \limits_C [\bigintsss \limits^{A,B} (FA \otimes GB \otimes \theta(A \otimes B,C), EC ]
	\<\simeq
	\commentaire{\ref{eq2}}
	\bigintsss \limits_C  \bigintsss \limits_A [FA, [ \bigintsss \limits^B GB \otimes \theta(A \otimes B,C), EC ]	
	\<\simeq
	\commentaire{Commutativity of ends}
	\bigintsss \limits_A  \bigintsss \limits_C [FA,[  \bigintsss \limits^B GB \otimes \theta(A \otimes B,C), EC ]]
	\<\simeq
	\commentaire{\ref{eq2}}
	\bigintsss \limits_A [FA, \bigintsss \limits_{B,C}  [ \theta(A \otimes B,C),[GB, EC ]]]
	\<\simeq
	\commentaire{}
	[A,[B,C]]
\end{deriv}	
 Thus the erosion of two binary matrices $F,G$ depicted by $F \ast^\prime G$ can be written as:
\begin{equation}
(F \ast^\prime G)[r,s]=\bigwedge (F[m,n] \wedge G[p,q])
\end{equation}
where $r= m-p $ and $s=n-q$.
The erosion of two gray-scaled images can be extracted by using max-plus semiring,
\begin{equation}
(F \ast^\prime G)[r,s]=\min (F[m,n]-G[p,q])
\end{equation}
where $r= m-p $ and $s=n-q$.

\section{conclusion}
Category theory provides an abstract unified framework for almost all aspects of mathematics. This is the first research conducted on creating a unified definition for two fundamental morphological operations of dilation and erosion. We have unified morphological operations appeared in contrasting situations like binary, gray-scaled and fuzzy operators into a uniform definition that can be extended to some new variations with using different semirings. 

An interesting horizon for the future research on category theory and mathematical morphology can be imagined by *-autonomous categories and mathematical morphology. *-autonomous categories are categorical representation of linear logic which has been a major research area. Many models for linear logic such as Petri nets and game semantics has been suggested but none of them is satisfying. Mathematical morphology is claimed to be a model of linear logic \cite{van2007modal}. The authors have shown that every derivable formula in linear logic should be a model of mathematical morphology but the reverse is an open problem. Conducting research on *-autonomous which are just symmetric monoidal categories with an involution object and morphological operations will help to shed the light over relation of linear logic and mathematical morphology.

\bibliographystyle{abbrv} 
\bibliography{mybib10}

\begin{thebibliography}{10}

\bibitem{borceux1994handbook}
F.~Borceux.
\newblock {\em Handbook of Categorical Algebra: Volume 1, Basic Category
  Theory}.
\newblock Cambridge Textbooks in Linguistics. Cambridge University Press, 1994.

\bibitem{borceux1994handbook2}
F.~Borceux.
\newblock {\em Handbook of Categorical Algebra: Volume 2, Basic Category
  Theory}.
\newblock Cambridge Textbooks in Linguis. Cambridge University Press, 1994.

\bibitem{borceux1994handbook3}
F.~Borceux.
\newblock {\em Handbook of Categorical Algebra: Volume 3, Basic Category
  Theory}.
\newblock Cambridge Textbooks in Linguis. Cambridge University Press, 1994.

\bibitem{day1970closed}
B.~Day.
\newblock On closed categories of functors.
\newblock In {\em Reports of the Midwest Category Seminar IV}, pages 1--38.
  Springer, 1970.

\bibitem{day1969enriched}
B.~J. Day and G.~M. Kelly.
\newblock Enriched functor categories.
\newblock In {\em Reports of the Midwest Category Seminar III}, pages 178--191.
  Springer, 1969.

\bibitem{de1998fuzzy}
B.~De~Baets.
\newblock A fuzzy morphology: a logical approach.
\newblock In {\em Uncertainty analysis in engineering and sciences: fuzzy
  logic, statistics, and neural network approach}, pages 53--67. Springer,
  1998.

\bibitem{erne2004adjunctions}
M.~Ern{\'e}.
\newblock Adjunctions and galois connections: Origins, history and development.
\newblock In {\em Galois connections and Applications}, pages 1--138. Springer,
  2004.

\bibitem{haralick1987image}
R.~M. Haralick, S.~R. Sternberg, and X.~Zhuang.
\newblock Image analysis using mathematical morphology.
\newblock {\em IEEE transactions on pattern analysis and machine intelligence},
  (4):532--550, 1987.

\bibitem{heijmans1990algebraic}
H.~J. Heijmans and C.~Ronse.
\newblock The algebraic basis of mathematical morphology i. dilations and
  erosions.
\newblock {\em Computer Vision, Graphics, and Image Processing},
  50(3):245--295, 1990.

\bibitem{im1986universal}
G.~B. Im and G.~M. Kelly.
\newblock A universal property of the convolution monoidal structure.
\newblock {\em Journal of Pure and Applied Algebra}, 43(1):75--88, 1986.

\bibitem{kelly1982basic}
M.~Kelly.
\newblock {\em Basic concepts of enriched category theory}, volume~64.
\newblock CUP Archive, 1982.

\bibitem{lawvere1963functorial}
F.~W. Lawvere.
\newblock Functorial semantics of algebraic theories.
\newblock {\em Proceedings of the National Academy of Sciences of the United
  States of America}, 50(5):869, 1963.

\bibitem{leinster2014basic}
T.~Leinster.
\newblock {\em Basic category theory}, volume 143.
\newblock Cambridge University Press, 2014.

\bibitem{loregian2015co}
F.~Loregian.
\newblock This is the (co) end, my only (co) friend.
\newblock {\em arXiv preprint arXiv:1501.02503}, 2015.

\bibitem{opac-b1078351}
S.~Mac~Lane.
\newblock {\em Categories for the working mathematician}.
\newblock Graduate texts in mathematics. Springer-Verlag, New York, 1978.

\bibitem{memar2017}
H.~Memarzadeh~sharifipour, B.~Yousefi, and X.~Maldague.
\newblock Skeletonization and reconstruction based on graph morphological
  transformations.
\newblock {\em Advanced Infrared Technology and Applications Conference}, 2017.

\bibitem{riehl2017category}
E.~Riehl.
\newblock {\em Category theory in context}.
\newblock Courier Dover Publications, 2017.

\bibitem{serra1983image}
J.~Serra.
\newblock {\em Image analysis and mathematical morphology}.
\newblock Academic Press, Inc., 1983.

\bibitem{shmuely1974structure}
Z.~Shmuely.
\newblock The structure of galois connections.
\newblock {\em Pacific Journal of Mathematics}, 54(2):209--225, 1974.

\bibitem{van2007modal}
J.~van Benthem and G.~Bezhanishvili.
\newblock Modal logics of space.
\newblock In {\em Handbook of spatial logics}, pages 217--298. Springer, 2007.

\end{thebibliography}

\end{document}